\documentclass{amsart} %pdflateX 日本語だめ
\usepackage{amsmath}
\usepackage{amssymb} 
\usepackage{amscd} 
\usepackage{graphicx}  %\usepackage[dvipdfmx]{graphicx}
\usepackage{epsfig,comment}
\usepackage[all]{xy} 
\usepackage{marvosym}
\usepackage{float}
\usepackage{xcolor}
\usepackage{ascmac}
\usepackage[margin=1.2in]{geometry} %横に広げる
\usepackage{hyperref}

\hypersetup{
	colorlinks=true,
	linktoc=all,
    	linkcolor={red!50!black},
   	citecolor={blue!50!black},
  	urlcolor={blue!80!black}
	} %%% This tones down the bright green hyperlink box and its other colors...

\definecolor{dg}{rgb}{0.1,0.4,0.1}
\renewcommand{\labelenumi}{ $(\arabic{enumi})$ }

\newtheorem{theorem}{Theorem}[section]
\newtheorem{lemma}[theorem]{Lemma}

\newtheorem{proposition}[theorem]{Proposition}
\newtheorem{corollary}[theorem]{Corollary}
\newtheorem{claim}[theorem]{Claim}

\newtheorem{remark}[theorem]{Remark}
\newtheorem{question}[theorem]{Question}

\theoremstyle{definition}
\newtheorem{definition}[theorem]{Definition}
\newtheorem{example}[theorem]{Example}

\newtheorem*{thm_existence_persistent_elements}{Theorem~\ref{existence_persistent_elements}}
\newtheorem*{thm_minimizer}{Theorem~\ref{minimizer}}
\newtheorem*{thm_characterization}{Theorem~\ref{characterization}}
\newtheorem*{thm_complement}{Theorem~\ref{complement}}

\numberwithin{equation}{section}
\numberwithin{figure}{section}
\numberwithin{table}{section}

\newcommand{\QED}[1]{\hspace*{\fill} $\square$(#1)\newline}
\renewcommand{\(}{\textup{(}}
\renewcommand{\)}{\textup{)}}

\begin{document}
\baselineskip 13pt
%%%%%%%%%%%%

\title[Dehn filling and the knot group]{Dehn filling and the knot group II: \\Ubiquity of persistent elements}

\author[T.Ito]{Tetsuya Ito}
\address{Department of Mathematics, Kyoto University, Kyoto 606-8502, JAPAN}
\email{tetitoh@math.kyoto-u.ac.jp}
%\thanks{The first named author has been partially supported by JSPS KAKENHI Grant Number JP15K17540 and JP16H02145.}

\author[K. Motegi]{Kimihiko Motegi}
\address{Department of Mathematics, Nihon University, 
3-25-40 Sakurajosui, Setagaya-ku, 
Tokyo 156--8550, Japan}
\email{motegi.kimihiko@nihon-u.ac.jp}
%\thanks{The second named author has been partially supported by JSPS KAKENHI Grant Number JP19K03502, 21H04428 and Joint Research Grant of Institute of Natural Sciences at Nihon University for 2022. }

\author[M. Teragaito]{Masakazu Teragaito}
\address{Department of Mathematics Education, Hiroshima University, 
1-1-1 Kagamiyama, Higashi-Hiroshima, 739--8524, Japan}
\email{teragai@hiroshima-u.ac.jp}
%\thanks{The third named author has been partially supported by JSPS KAKENHI Grant Number JP20K03587.}

\dedicatory{}

\begin{abstract}   
Let $K$ be a nontrivial knot in $S^3$. 
We say that an element of the knot group $G(K)$ is \textit{persistent} if  it remains nontrivial under all nontrivial Dehn fillings. 
Such elements exist for every nontrivial knot. 
Indeed, Property P is equivalent to the statement that the meridian of $K$ is a persistent element, and this represents the first instance of such elements.  
Building on the solution to the Property P conjecture due to Kronheimer and Mrowka, we show that every nontrivial knot group admits infinitely many persistent elements with pairwise disjoint automorphic orbits, none of which contains a power of the meridian. 
We then develop this further to show that for a broad class of hyperbolic knots -- namely those admitting no surgery whose resulting manifold has torsion in its fundamental group -- persistent elements are not rare curiosities, but rather structurally pervasive in $G(K)$. 
This is reflected in the following two properties: 

(i) Every subgroup of $G(K)$ that is not contained in the normal closure of a peripheral element contains persistent elements.

(ii) Persistent elements exist outside every proper subgroup of $G(K)$.
\end{abstract}

\maketitle
\tableofcontents
%%%%%%%%%%%%%%%%%%%%%%%%%%%%%
\section{Introduction}
\label{Introduction}

\subsection{Background}
\label{subsection:background}

We have explored the algebraic effects of Dehn filling, focusing on how elements of the knot group behave under Dehn fillings \cite{IMT_Magnus,
IMT_realization,IMT_MSJ_KMS}. 
To be precise, we have introduced a set-valued function 
\[
\mathcal{S} \colon G(K)  \to 2 ^{\mathbb{Q}} \ \textrm{as}
\]
\[
\mathcal{S}(g) = \{ r \in \mathbb{Q} \mid g \  \textrm{becomes trivial after $r$--Dehn filling} \} \subset \mathbb{Q}.
\]

Let us collect some known properties of this function. 

\begin{itemize}
\item[-] \textbf{Conjugate invariance} --\ $\mathcal{S}$ is a class function, i.e. $\mathcal{S}(aga^{-1}) =\mathcal{S}(g)$ for any elements $a$ and $g$.

\medskip

\item[-] \textbf{Finiteness property } --\ $\mathcal{S}(g)$ is finite for all nontrivial elements $g \in G(K)$ for hyperbolic knots (Groves-Manning \cite{GM}, Osin \cite{Osin}). 
More precisely, $\mathcal{S}(g)$ is finite for all nontrivial elements $g \in G(K)$ if and only if $K$ is a prime, non-cabled knot \cite{IMT_residually_perfect}. 

\medskip

\item[-] \textbf{Injectivity of universal Dehn filling homomorphism} --\ $\mathcal{S}(g) = \mathbb{Q}$ if and only if $g$ is the identity element, in other words, the universal Dehn filling homomorphism
\[
\Phi \colon G(K) \to \prod_{r \in \mathbb{Q}} \pi_1(K(r))
\]
is injective \cite{IMT_residually_perfect}. 

\medskip

\item[-] \textbf{Realization property} --\ If $K$ is a hyperbolic knot without Seifert surgery or reducing surgery, then for any given finite subset $\mathcal{R} \subset \mathbb{Q}$, 
we have an element $g \in [G(K), G(K)]$ such that $\mathcal{S}(g) = \mathcal{R}$  \cite{IMT_realization}.  
This property reflects an algebraic independence of Dehn fillings.  
\end{itemize}

The size of $\mathcal{S}(g)$ is, by definition, a measure of the strength of $g$ against Dehn fillings. 
It is very hard to explicitly determine $\mathcal{S}(g)$ for a given nontrivial element $g \in G(K)$. 
Following Realization property, $\mathcal{S}(g)$ can attain a given finite subset even when its cardinal is huge. 
On the contrary,  
in this article, we focus upon elements $g$ with $\mathcal{S}(g) = \emptyset$, i.e. elements unaffected by any nontrivial Dehn filling. 

\subsection{Motivation}
\label{subsection:motivation}

The Property P conjecture, 
proposed by Bing and Martin \cite{BingMartin},  
asserts that the result $K(r)$ of $r$--Dehn surgery on a nontrivial knot $K$ is never simply connected for all $r \in \mathbb{Q}$. 
Since torus knots were known to satisfy the Property P \cite{Mos},  
the Cyclic Surgery Theorem \cite{CGLS} reduces the conjecture to the specific case of $r = \pm 1$. 
Kronheimer and Mrowka \cite{KM} prove an existence of a non-abelian $\mathrm{SU}(2)$--representation of 
$\pi_1(K(1))$ for any nontrivial knot, which settles the Property P conjecture in the positive. 
They then strengthen this result from gauge theoretic point of view, 
proving $\pi_1(K(r))$ admits such a representation for any rational number $r \in [0, 2]$. 
Baldwin-Sivek \cite{BS} and Baldwin-Li-Sivek-Ye \cite{BLSY} prove $\pi_1(K(4))$ and $\pi_1(K(3))$ also admit $\mathrm{SU}(2)$--non-abelian representations. 
These result concern how far is $\pi_1(K(r))$ compared to the trivial group.

This article is also motivated by Property P in a different perspective based upon the following reformulation of Property P.  
Since the meridian normally generates the knot group,
if it becomes trivial in $\pi_1(K(r))$, then  $\pi_1(K(r)) = \{ 1 \}$.  
This immediately enables us to restate Property P. 

\begin{theorem}[Reformulation of Property P  \cite{KM}]
\label{PropertyP_rephrase}
For any nontrivial knot, its meridian $\mu$ survives under all nontrivial Dehn fillings, i.e. $\mathcal{S}(\mu) = \emptyset$. 
\end{theorem}

Motivated by this, we introduce the notion of persistence: a nontrivial element of the knot group is said to be  {\em persistent} if it survives every nontrivial Dehn filling.
The persistent elements form a “rigid core” of the knot group, independent of all nontrivial Dehn fillings and the peripheral relations they induce. 
Our goal is to understand the structure of this set. This leads to the following fundamental question.

\begin{question}
Are persistent elements rare or abundant in a knot group?
\end{question}
 
The existence of many persistent elements reveals a previously underappreciated form of deep group-theoretic rigidity inherent in knot groups
-- one that is invisible at the level of any individual Dehn fillings but emerges from their collective behavior.

\medskip

\subsection{Results}
\label{subsection:results}

Since any automorphic image of the meridian also normally generates the knot group, it is persistent as well. 
We say that two elements are \textit{equivalent} if one is the automorphic image of the other.  
Given an element $g \in G(K)$, 
by the {\em automorphic orbit} of $g$ we mean the set of all automorphic images of $g$, 
i.e., the set $\{ \varphi(g)\}$ where $\varphi$ runs over all automorphisms of $G(K)$. 

\begin{theorem}[non-meridional persistent elements]
\label{existence_persistent_elements}
For every nontrivial knot, 
its group contains infinitely many persistent elements with disjoint automorphic orbits, 
none of which contains a power of the meridian.
\end{theorem}

We used this result in \cite{IMT_cyclic} to show that the knot group admits a cyclic subgroup whose nontrivial elements are persistent if and only if the knot has no finite surgery. 

\begin{remark}
\label{auto_persistent}\
\begin{enumerate}
\item
As we mentioned above, for any inner automorphism $\psi$, 
$\mathcal{S}(\psi(g)) = \mathcal{S}(g)$. 
\item
If $K$ is prime, 
then any automorphism of $G(K)$ is induced by a homeomorphism of $E(K)$ up to conjugation \cite{Tsau2}.
Hence, for equivalent elements $g,\ g' \in G(K)$, 
$\mathcal{S}(g)= \mathcal{S}(g')$ if $K$ is not amphicheiral, 
and $\mathcal{S}(g)= \pm \mathcal{S}(g')$ if $K$ is amphicheiral.
In particular, an element equivalent to a persistent element is persistent. 
\item
On the other hand, if $K$ is not prime, 
it may happen that $\mathcal{S}(g) \neq \pm \mathcal{S}(g')$ even when $g$ and $g'$ are equivalent. 
Actually, a persistent element $g$ may be equivalent to a non-persistent element  $g'$ in $G(K)$, 
i.e.  $\mathcal{S}(g)=  \emptyset$, while $\mathcal{S}(g') \ne \emptyset$; 
see Proposition~\ref{persistent_auto}
\end{enumerate}
\end{remark}

\medskip

Developing Theorem~\ref{existence_persistent_elements}  further, 
we show that persistent elements are ubiquitous among subgroups of the knot group.
The theorem applies to a broad class of hyperbolic knots, namely those admitting no surgery whose resulting manifold has torsion in its fundamental group.

\begin{definition}
\label{D}
Let $K$ be a nontrivial knot in $S^3$. 
To each subgroup $H \subset G(K)$, 
we define 
\[
\mathcal{S}_{\cap}(H) = \bigcap_{h \in H} \mathcal{S}(h) \subset \mathbb{Q}. 
\]
\end{definition}

If the result $K(r)$ has the fundamental group with nontrivial torsion, then we call such a surgery a \textit{torsion surgery}. 
It is known that a torsion surgery is either a finite surgery or a reducing surgery.  
In what follows, we focus on hyperbolic knots that admit no torsion surgeries; we call such knots {\em torsion-free hyperbolic knots} for simplicity.

For a given subgroup $H$ of $G(K)$, 
we call an element $h_0 \in H$ a \textit{minimizer} of $H$ if $\mathcal{S}(h_0) = \mathcal{S}_{\cap}(H)$ holds. 

Using quantitative information
provided by stable commutator length, together with geometric control arising from hyperbolic structures on knot complements and their Dehn fillings, 
we may prove: 

\begin{theorem}[Existence of minimizer]
\label{minimizer}
Let $K$ be a torsion-free hyperbolic knot.  
Then every nontrivial subgroup $H$ of $G(K)$ has a minimizer. 
\end{theorem}

We apply Theorem~\ref{minimizer} to obtain a complete characterization of subgroups that contain persistent elements. 
Obviously,  not every subgroup of $G(K)$ contains a persistent element. 
Let $\gamma$ be a \textit{slope element} represented by a  simple closed curve on $\partial E(K)$ of slope $r \in \mathbb{Q}$. 
Then its normal closure is denoted by $\langle\!\langle\gamma \rangle\!\rangle$.  
(Both $\gamma$ and $\gamma^{-1}$ correspond to the slope $r$, 
and $\langle\!\langle\gamma \rangle\!\rangle = \langle\!\langle\gamma^{-1} \rangle\!\rangle$, 
so it is reasonable to denote it by $\langle\!\langle r \rangle\!\rangle$.)
Then each element in $\langle\!\langle r \rangle\!\rangle$ becomes trivial after $r$--Dehn filling, 
so every subgroup of $\langle\!\langle r \rangle\!\rangle$ never contain persistent elements for all $r \in \mathbb{Q}$.

\begin{theorem}
\label{characterization}
Let $K$ be a torsion-free hyperbolic knot, 
and $H$ a subgroup of $G(K)$. 
Then the following three conditions are equivalent. 

\begin{enumerate}
\item 
$H$ contains a persistent element. 
\item
$H$ is not contained in $\langle\!\langle r \rangle\!\rangle$ for any $r \in \mathbb{Q}$.  
\item 
$\mathcal{S}_{\cap}(H) = \emptyset$.
\end{enumerate}
\end{theorem}

Thus subgroups which obviously do not admit a persistent element are the only subgroups containing no persistent elements. 

Theorem~\ref{characterization} has several applications. 
If $H \subset \langle\!\langle r \rangle\!\rangle$, then $H$ vanishes in $\pi_1(K(r))$, and vice versa. 
So $(1) \Leftrightarrow (2)$ implies:

\begin{corollary}
Let $K$ be a torsion-free hyperbolic knot. 
For any subgroup $H$ of $G(K)$, we have the following alternative: 
\begin{itemize}
\item
$H$ totally vanishes for some nontrivial Dehn filling, or 
\item
$H$ admits a persistent element.
\end{itemize}
\end{corollary}

As an application of Theorem~\ref{characterization}, we have: 

\begin{example}[Seifert surface subgroup]
\label{Seifert_surface}
Let $K$ be a torsion-free hyperbolic knot. 
Let $S \subset E(K)$ be an incompressible Seifert surface of $K$. 
Take $\lambda \in \pi_1(\partial S)$ and any nontrivial element $\alpha \in \pi_1(S)$ which is not conjugate into $\pi_1(\partial S)$ in $\pi_1(S)$. 
Then $\mathcal{S}(\lambda) = \{ 0 \}$, but $\mathcal{S}(\alpha) \not\ni 0$ \cite{GabaiIII}. 
Thus $\mathcal{S}(\lambda) \cap \mathcal{S}(\alpha) = \emptyset$. 
Then Corollary~\ref{characterization} ($(2) \Rightarrow (1)$) shows that 
$\langle \lambda, \alpha \rangle$ admits a persistent element. 
Hence, the Seifert surface subgroup $\pi_1(S)$ contains a persistent element. 
\end{example}

\begin{remark}
\label{Seifert_surface_cyclic}
Example~\ref{Seifert_surface} does not hold if $K$ has a cyclic surgery. 
Actually, if $K$ admits a cyclic surgery slope $r$,  
then every element of the commutator subgroup vanishes for $r$--Dehn filling. 
In particular, the Seifert surface subgroup ($\subset [G(K), G(K)]$) 
never contains a persistent element. 
\end{remark}

\medskip

Generically, when we choose $g_1, \dots, g_n \in G(K)$ arbitrarily, 
we may expect $\mathcal{S}(g_1) \cap \cdots \cap \mathcal{S}(g_n) = \emptyset$.  
So, generically, non-cyclic subgroups are expected to admit persistent elements.

Recall that every nontrivial not group $G(K)$ is fully residually finite \cite{Hem_residual_finite}, 
i.e. for a given finitely many nontrivial elements $g_1, \dots, g_n \in G(K)$, 
we may find a finite index subgroup $H$ which does not contain $g_1, \dots, g_n$. 
Thus $G(K)$ is rich in finite index subgroups. 
On the other hand, 
for any hyperbolic knot without finite surgery, 
$\langle\!\langle r \rangle\!\rangle$ has infinite index in $G(K)$ for $r \in \mathbb{Q}$, 
and hence no finite index subgroup is contained in $\langle\!\langle r \rangle\!\rangle$. 
Thus Theorem~\ref{characterization} implies 

\begin{corollary}[Finite index subgroup]
\label{finite_index}
Let $K$ be a torsion-free hyperbolic knot. 
Every finite index subgroup of $G(K)$ contains a persistent element. 
\end{corollary}

The next corollary puts a constraint on subgroups contained in $\displaystyle\bigcup_{r \in \mathbb{Q}} \langle\!\langle r \rangle\!\rangle$. 

\begin{corollary}
\label{structure_subgroup}
Let $K$ be a torsion-free hyperbolic knot. 
Every subgroup $H$ contained in $\displaystyle\bigcup_{r \in \mathbb{Q}} \langle\!\langle r \rangle\!\rangle$ 
lies in $\langle\!\langle r \rangle\!\rangle$ for some $r \in \mathbb{Q}$. 
\end{corollary}

Let us turn to look at the complement of subgroups of $G(K)$. 
The next result asserts that 
we may find a persistent element in out side of any proper subgroup of $G(K)$. 

\begin{theorem}
\label{complement}
Let $K$ be a torsion-free hyperbolic knot. 
Then for any proper subgroup $N$ of $G(K)$, 
there exists a persistent element $g$ in $G(K) - N$.
\end{theorem}

As a direct consequence of this, we have: 

\begin{corollary}
\label{cor:complement}
Let $K$ be a torsion-free hyperbolic knot. 
Then there is no proper subgroup of $G(K)$ which contains all persistent elements.  
\end{corollary}

\medskip

\subsection{Persistent elements and pseudo-meridians}
\label{subsection:pseudo-meridian}

We say that a nontrivial element $g$ is a \textit{pseudo-meridian} if 
it normally generates $G(K)$, but it is not an automorphic image of a meridian \cite{Tsau1,SWW}.  
It is conjectured that every nontrivial knot admits a pseudo-meridian \cite{SWW}. 
However, this conjecture is widely open. 

Suppose that $G(K)$ admits a pseudo-meridian $g$. 
Then, by definition, a meridian can be expressed as finite product of some conjugate of $g$, 
and hence if $g$ becomes trivial in $\pi_1(K(r))$, 
then so does the meridian.  
Following Theorem~\ref{PropertyP_rephrase} $r = \infty$, 
i.e. $g$ is a persistent element. 
In Section~\ref{p-m}, we will discuss a relationship between persistent elements and pseudo-meridians.

\medskip

\section{Persistent elements in a knot group}

In this section we will establish: 

\begin{thm_existence_persistent_elements}
For every nontrivial knot, 
its group contains infinitely many persistent elements with disjoint automorphic orbits, 
none of which contains a power of the meridian.
\end{thm_existence_persistent_elements}

\subsection{Baumslag-Solitar relators}
\label{subsection:BS}

Let $G$ be a group. 
Assume that $x, y \in G$ satisfies $x^n = y x^m y^{-1}$ for nonzero integers $m$ and $n$. 
We call the relation $x^n = y x^m y^{-1}$ the \textit{Baumslag-Solitar relation}. 

The following result due to Shalen \cite{Shalen}, 
which generalizes the previous result \cite{JS}, 
asserts that the Baumslag-Solitar relation cannot hold in a non-degenerate way in the fundamental group of an orientable $3$--manifold. 

\begin{proposition}[\cite{Shalen}]
\label{BS_Shalen}
Let $M$ be an orientable $3$--manifold. 
Suppose that there are elements $x, y \in \pi_1(M)$ and nonzero integers $m$ and $n$ such that 
$x^n = y x^m y^{-1}$. 
Then either 
\begin{enumerate}
\item
$x$ has a finite order, or 
\item
$m = \pm n$
\end{enumerate}
\end{proposition}

Following Proposition~\ref{BS_Shalen}, 
the relation $x^n = y x^m y^{-1}$ is very restrictive in $3$--manifold groups. 
We call a word of type $x^{-n} y x^m y^{-1}$ the \textit{Baumslag-Solitar relator}.  

For each $r \in \mathbb{Q}$, $r$--Dehn filling induces a natural epimorphism 
\[
p_r \colon G(K) \to  \pi_1(K(r)) = G(K) / \langle\!\langle  r \rangle\!\rangle. 
\]

A simple application of Proposition~\ref{BS_Shalen} leads us:

\begin{lemma}
\label{BS_relation}
Let $x$ and $y$ be nontrivial elements in $G(K)$ such that $p_r(x)$ is nontrivial in $\pi_1(K(r))$ unless $r$ is a torsion surgery slope.  
Put $g = x^{-n}yx^{m}y^{-1} \in G(K)$ $(m \ne \pm n)$. 
Then $p_r(g)$ is also nontrivial in $\pi_1(K(r))$ unless $r$ is a torsion surgery slope. 
Furthermore, if $x$ is the meridian $\mu$ of the knot $K$,  
then $p_r(g)$ is non trivial in $\pi_1(K(r))$ unless $r$ is a finite surgery slope.
\end{lemma}

\begin{proof}
Assume that $r$ is not a finite surgery slope, 
and that 
$p_r(g) = p_r(x^{-n}yx^{m}y^{-1}) = 1 \in \pi_1(K(r))$.   
  
If $p_r(y) =1 \in \pi_1(K(r))$, 
then $p_r(g) = p_r(x)^{-n} p_r(x)^m = p_r(x)^{m-n}  = 1 \in \pi_1(K(r))$. 
By the assumption, $m - n \ne 0$. 
If $|m - n| =1$, then $p_r(x) = 1$ in $\pi_1(K(r))$, 
contradicting the choice of $x$. 
So $|m - n| \ge 2$, and $p_r(x)$ is a torsion element in $\pi_1(K(r))$, a contradiction.

If $p_r(y) \ne 1 \in \pi_1(K(r))$, 
then we have the Baumslag-Solitar relation 
\[
p_r(x)^{n} = p_r(y)p_r(x)^{m}p_r(y)^{-1}\ \textrm{in}\  \pi_1(K(r))\  
\textrm{with}\ m \ne \pm n.
\] 
It follows from Proposition~\ref{BS_Shalen} that $p_r(x)$ is a torsion element in $\pi_1(K(r))$, 
contradicting the initial assumption.

Now we take $x$ as the meridian $\mu$.  
Then it remains nontrivial after any nontrivial Dehn filling on $K$ (Theorem~\ref{PropertyP_rephrase}). 
Assume that $p_r(g) = p_r(\mu^{-n}y \mu^{m}y^{-1}) = 1$ in $\pi_1(K(r))$. 
If $p_r(y) =1 $ in $\pi_1(K(r))$, 
then $p_r(g) =  p_r(\mu)^{-n}  p_r(\mu)^m =  p_r(\mu)^{m-n}  = 1 \in \pi_1(K(r))$. 
By the assumption, $m - n \ne 0$. 
If $|m - n| =1$, then $p_r(\mu) = 1$ in $\pi_1(K(r))$, a contradiction. 
So $|m - n| \ge 2$, 
and $p_r(\mu)$ is a torsion element in $\pi_1(K(r))$, 
i.e. $p_r(\mu)^{m-n} =1$. 
This implies that $r$ is a finite surgery slope or a reducing surgery slope. 
\cite[Proposition~3.1]{IMT_realization} excludes the latter possibility. 

If $p_r(y) \ne 1 \in \pi_1(K(r))$, 
then we have the Baumslag-Solitar relation 
$p_r(\mu)^{n} = p_r(y)p_r(\mu)^{m}p_r(y)^{-1}$ in $\pi_1(K(r))$ with $m \ne \pm n$. 
It follows from Proposition~\ref{BS_Shalen} that $p_r(\mu)$ is a torsion element in $\pi_1(K(r))$. 
Hence, $r$ is a finite surgery slope as above.
\end{proof}

\medskip

Suitably choosing parameters $(m, n)$ in the Baumslag-Solitar relator, 
we obtain:

\begin{theorem}
\label{persistent_BS}
Let $K$ be a knot in $S^3$ which is not a trefoil knot. 
For any nontrivial element $y \in G(K)$, 
let $g = \mu^{n - m}y \mu^{m}y^{-1} \in G(K)$ \($m > 1$ and $|n| =1, 2$\). 
Then $g$ is a persistent element. 
\end{theorem}

\begin{proof}
Following Lemma~\ref{BS_relation} $p_s(g) \ne 1$ for all but finite surgery slopes. 
Let $s$ be a finite surgery slope of $K$ and write $s = a/b$ $(a > 0)$. 
Assume that $p_{a/b}(g) = 1$ in $\pi_1(K(s))$. 
Note that $[g] = n [\mu] \in H_1(E(K)) = G(K)/ [G(K), G(K)]$ ($|n| =1, 2$).  
Thus $[p_{a/b}(g)] = n [\mu] \in H_1(K(a/b)) \cong \mathbb{Z}_a$, 
where we use $[\mu]$ to denote the image of the generator $[\mu] \in H_1(E(K))$. 
On the other hand, since $p_{a/b}(g) = 1 \in \pi_1(K(a/b))$, 
$[p_{a/b}(g)] =0 \in H_1(K(a/b))$, 
i.e. $n [\mu] = 0$ in $H_1(K(a/b)) \cong \mathbb{Z}_a$. 
Hence $n$ is a multiple of $a > 0$. 
Our assumption $|n| = 1, 2$ implies $a = 1$ or $2$, 
and $|H_1(K(a/b))|  = 1$ or $2$. 
Then the lemma below shows that $K$ is a trefoil knot. 
\end{proof}

\begin{lemma}
\label{spherical_HS3}
Let $K$ be a nontrivial knot. 
Assume that $r$ is a finite surgery slope with $|H_1(K(r))|= 1$ or $2$ for some $r \in \mathbb{Q}$. 
Then $K$ is a trefoil knot.  
\end{lemma}

\begin{proof} 
By taking mirror image, we may assume $r > 0$, where $r = 1/b$ or $2/b$ for some integer $b$. 
Thus $r \le 2$. 
Furthermore, since $|\pi_1(K(r))| < \infty$, $K(r)$ is an L-space, 
and hence \cite{OS_rational} shows that $2g(K) -1 \le r \le 2$, i.e. $2g(K) \le 3$. 
Hence $g(K) = 1$.  
Any L-space knot is known to be fibered \cite{Ni}, 
and hence $K$ must be a trefoil knot or the figure-eight knot. 
The figure-eight knot does not admit a finite surgery, and thus $K$ is a trefoil knot. 
\end{proof}

\begin{remark}
\label{persistent_trefoil}
Following \cite{SWW} a trefoil knot has infinitely many, mutually non-equivalent pseudo-meridians. 
They are persistent elements, each of which represents a homological generator. 
\end{remark}

\medskip

\subsection{Proof of Theorem~\ref{existence_persistent_elements}}
\label{subsection*proof}

Let us prove Theorem~\ref{existence_persistent_elements}. 
We divide into four cases according as $K$ is 
a cabled knot (Case 1), a hyperbolic knot (Case 2), 
a prime, non-cabled satellite knot (Case 3), or a non-prime knot (Case 4). 
In the following, we prove there are infinitely many, 
mutually non-equivalent persistent elements each of which represents the homological generator. 

This immediately implies that the knot group contains infinitely many persistent elements whose automorphic orbits are disjoint, 
and each automorphic orbit does not contain a power of the meridian.
Actually, 
if automorphic images of $g$ and $h$ intersect, then $g$ and $h$ are equivalent, a contradiction. 
Further, since each persistent element we construct is a homological generator, 
if it is not equivalent to the meridian, then it is not equivalent to any power of the meridian neither. 

\medskip 

\noindent
\textbf{Case 1}\ :\  $K$ is a cabled knot. 
If $K$ is a cable of the trivial knot, i.e. it is a torus knot, 
then Silver-Whitten-Williams \cite{SWW} demonstrates that 
$G(K)$ has infinitely many, non-equivalent pseudo-meridians.  
While, if the companion knot is nontrivial, 
then Dutra \cite{Dutra} shows that 
$G(K)$ has infinitely many, mutually non-equivalent pseudo-meridians. 

As noted in Subsection~\ref{subsection:pseudo-meridian}, 
any pseudo-meridian is a persistent element, 
which represents a generator of $H_1(E(K))$. 

\medskip 

In Cases 2 and 3, 
we first show that there are infinitely many, 
mutually ``non-conjugate'' persistent elements each of which represents the homological generator. 
Shortly after that we will improve``being non-conjugate'' to ``being non-equivalent''. 

\medskip 

\noindent
\textbf{Case 2}\ :\ $K$ is a hyperbolic knot. 
Let us take $g_m = \mu^{1-m}y \mu^m y^{-1}$, where $\mu y \ne y \mu$. 
Then Theorem~\ref{persistent_BS} shows that $g_m$ is a persistent element for all integers $m > 1$. 
Furthermore, apply Lemma~\ref{non-conjugate} below to obtain infinitely many, 
mutually non-conjugate persistent elements each of which represents the homological generator. 

\medskip

\noindent
\textbf{Case 3}\ :\ $K$ is a prime, non-cabled satellite knot. 
Consider the torus decomposition of $E(K)$ and let $X$ be the outermost piece, 
namely the decomposing piece which contains $\partial E(K)$. 
Note that the inclusion map induces a monomorphism $\pi_1(X) \to G(K)$, and we regard $\pi_1(X) \subset G(K)$. 
Since $K$ is a prime, non-cabled satellite knot, $X$ is hyperbolic (\cite[VI.3.4.Lemma]{JS}), 
in particular $\pi_1(X)$ is non-abelian. 
We take $y \in \pi_1(X)$ so that $[\mu, y] \ne 1$ in $\pi_1(X)$. 
Let us take $g_m = \mu^{1-m}y \mu^m y^{-1} \in \pi_1(X) \subset G(K)$. 
Then Theorem~\ref{persistent_BS} shows that $g_m$ is a persistent element for all integers $m > 1$. 
Following Lemma~\ref{non-conjugate} $g_m$ and $g_n$ $(m, n >1)$ are not conjugate in $\pi_1(X)$ when $m \ne n$.

\medskip

In cases 2 and 3, 
we have infinitely many, mutually non-conjugate persistent elements in $\pi_1(X)$. 
Now we show that thee are also infinitely many, mutually non-equivalent persistent elements in $\pi_1(X)$, 
and then we show these are non-equivalent in $G(K)$. 

By Mostow-Prasad rigidity $\mathrm{Out}(\pi_1(X)) = \mathrm{Aut}(\pi_1(X))/\mathrm{Inn}(\pi_1(X))$ is finite. 
So choose representatives $\varphi_1, \dots, \varphi_n$ of $\mathrm{Out}(\pi_1(X))$. 
Then for any automorphism $\varphi$ and any element $g \in \pi_1(X)$, 
$\varphi(g)$ is conjugate to $\varphi_i(g)$ for some $i$. 
Thus for each $g$, $\{ g, \varphi_1(g), \dots, \varphi_n(g) \}$ is the automorphic orbit of $g$,which consists of mutually non-conjugate elements. 
Since we have infinitely many, mutually non-conjugate persistent elements in $\pi_1(X)$, 
we have also infinitely many, mutually non-equivalent persistent elements $g_1, \dots$ in $\pi_1(X)$. 
Let us show that these elements are still non-equivalent in $G(K)$. 
Suppose that $g_m$ and $g_n$ $(m \ne n)$ are not equivalent in $\pi_1(X)$, but they are equivalent in $G(K)$ for some $m$ and $n$. 
Here we recall that we chose and fixed a base-point of $\pi_1(X) \subset G(K)$ on $\partial E(K)$.  
Then we have an automorphism $\varphi$ of $G(K)$ such that $\varphi(g_m) = g_n$. 
Since $\varphi$ is induced, up yo conjugation, 
 by a homeomorphism $f$ of $E(K)$ \cite[Corollary~4.2]{Tsau2}, 
up to isotopy which leave the base-point invariant, we may assume $f(X) = X$, and thus $\varphi(\pi_1(X)) = \pi_1(X)$. 
Hence $\varphi |_{\pi_1(X)}$ sends $g_m$ to $g_n$. 
This is a contradiction. 
Hence, we have infinitely many, 
mutually non-equivalent persistent elements in $G(K)$. 
Note that since $g_m = \mu^{1-m}y \mu^m y^{-1} \in \pi_1(X) \subset G(K)$, 
it represents the homological generator of $H_1(E(K))$.

\medskip 

\noindent
\textbf{Case 4}\ :\ 
$K$ is a non-prime knot. 
 
Then $E(K)$ has a decomposition $E(K) = X \cup E(k_1) \cup \cdots \cup E(k_n)$, 
where $k_i$ is a prime knot and $X$ is a composing space which is $[\textrm{disk with $n$--holes}] \times S^1$. 
Note that $\partial X = \partial E(K) \cup \partial E(k_1) \cup \cdots \cup \partial E(k_n)$ 
and a regular fiber $t$ in $\partial E(K)$ is a meridian of $K$. 
Thus $K(r) = E(k_1) \cup \cdots \cup E(k_n) \cup (X \cup_r (S^1 \times D^2))$, 
where $X \cup_r (S^1 \times D^2)$ is a $3$--manifold obtained from $X$ by gluing the solid torus along $\partial E(K) \subset \partial X$ with slope $r$.   
Note that $X \cup_r (S^1 \times D^2)$ is a boundary-irreducible Seifert fiber space for any $r \in \mathbb{Q}$. 
Therefore $\partial E(k_i)$ is incompressible in $K(r)$, and hence $\pi_1(E(k_i))$ injects into $\pi_1(K(r))$.

Then every nontrivial element of $G(k_i)$ is persistent. 
Let us choose $k = k_1$ and 
take a preferred meridian-longitude pair $(\mu_k, \lambda_k)$ of $k$ and elements $\mu_k\lambda_k^p \in G(k)$ ($p \in \mathbb{Z}$). 

Assume that we have an automorphism $\varphi$ of $G(k)$ such that $\varphi(\mu_k \lambda_k^p) = \mu_k \lambda_k^q$. 
Since $\varphi$ is induced by a homeomorphism $f$ of $E(k)$, 
$\varphi(\mu_k \lambda_k^p) = (\mu_k \lambda_k^p)^{\varepsilon}$ ($\varepsilon = \pm 1$). 
This means that $\mu_k \lambda_k^p$ and $\mu_k \lambda_k^q$ are equivalent in $G(k)$ if and only if $p = q$.  
We follow the argument in \cite{Dutra} to show that they are equivalent in $G(K)$ only if $p = q$ as well. 
Suppose that we have an automorphism $\Phi$ of $G(K)$ such that 
$\Phi(\mu_k \lambda_k^p) = \mu_k \lambda_k^q$.  
Then $\Phi$ is induced by a homotopy equivalence $f \colon E(K) \to E(K)$. 
By \cite[Theorem 14.6]{Jo}, 
$f$ may be deformed into $f'$ so that $f'|_X \colon X \to X$ is a homotopy equivalence, 
and $f'|_{\overline{E(K)-X}} \colon \overline{E(K)-X} \to \overline{E(K)-X}$ is a homeomorphism. 
Then for some permutation $\sigma$ of $\{ 1, \dots, n \}$, 
we have $f'(E(k_1)) = E(k_{\sigma(1)})$. 
We can write $\Phi(G(k)) = g G(k_{\sigma(1)}) g^{-1}$ for some (possibly trivial) element $g \in G(K)$. 
If $\sigma(1) = 1$ and $g \in G(k)$,  
then  $\Phi$ induces an automorphism $\psi = \Phi|_{G(k)} \colon G(k) \to G(k)$ such that $\psi(\mu_k \lambda_k^p) = \mu_k \lambda_k^q$. 
This implies that $p = q$. 
If $\sigma(1) \ne  1$ or $g \not \in G(k)$, 
then $ \mu_k \lambda_k^q$ is conjugated in $G(k)$ to an element of $\langle \mu_k  \rangle$ 
since $\varphi(\mu_k \lambda_k^p) = \mu_k \lambda_k^q \in G(k) \cap  g G(k_{\sigma(1)}) g^{-1}$. 
(Note that $\mu_{k_i} = \mu$ in $G(K)$ for all $i = 1, \dots, n$.) 
However, this can happen only when $q=0$; see \cite[Corollary~4.2]{Tsau2}. 
Hence,  if $p \ne q$, 
$\mu_k \lambda_k^p$ and $\mu_k \lambda_k^q$ are not equivalent (conjugate) in $G(K)$.

Finally we show that each $\mu_k \lambda_k^p$ ($p \in \mathbb{Z}$) is a homological generator. 
Recall that $\mu_k = \mu_{k_1} =\mu$ in $G(K)$, 
the meridian of $K$, and $\lambda_k = \lambda_1$ bounds a Seifert surface $S_{k_1}$ in $E(k_1)$. 
Hence $\mu_k \lambda^k$ is homologous to $\mu$ independent of $k$, which is a homological generator. 
\QED{Theorem~\ref{existence_persistent_elements}}

\medskip

We say that an element $g \in G(K)$ is \textit{peripheral} if it is conjugate into the {\em peripheral subgroup} $P(K) = i_*(\pi_1(\partial E(K)))$.

\begin{lemma}
\label{non-conjugate}
Let $X$ be a compact $3$--manifold whose boundary consists of tori. 
Assume that $\mathrm{int}(X)$ admits a complete hyperbolic structure of finite volume. 
Let $x, y$ be nontrivial elements in $\pi_1(X)$ with $[x, y] \ne 1$, 
and $x$ is peripheral. 
Then $x^{1-m} y x^m y^{-1}$ and  $x^{1-n} y x^n y^{-1}$ $(m, n >1)$ are not conjugate whenever $m \ne n$, 
and $x^{1-m} y x^m y^{-1}$ is non-peripheral, in particular not conjugate to any power of the meridian,  
except for at most integer $m > 1$. 
\end{lemma}

\begin{proof}
Let $\rho\colon  \pi_1(X) \to PSL_2(\mathbb{C})$ be a holonomy representation. 
Since $x$ is peripheral, 
$\rho(x)$ is  parabolic, 
and so 
by taking a conjugation 
we may assume that 
$\rho(x) = 
\begin{pmatrix}
	1 & \tau \\[2pt]
	0 & 1
\end{pmatrix}$ 
for some $0 \ne \tau \in \mathbb{C}$.  
Write 
$\rho(y) =
\begin{pmatrix}
	a & b \\[2pt]
	c & d
\end{pmatrix}$ 
for some $a, b, c, d \in \mathbb{C},\ (ad-bc = 1)$.  
Then \begin{eqnarray*}
\rho(x^{1-m} y x^m y^{-1}) 
=
\begin{pmatrix}
	1 & (1-m)\tau \\[2pt]
	0 & 1
\end{pmatrix}
\begin{pmatrix}
	a & b \\[2pt]
	c & d
\end{pmatrix}
\begin{pmatrix}
	1 & m\tau \\[2pt]
	0 & 1
\end{pmatrix}
\begin{pmatrix}
	d & -b \\[2pt]
	-c & a
\end{pmatrix}\\
=
\begin{pmatrix}
	(a-c(m-1)\tau)(d-cm\tau) - c(b-d(m-1)\tau)& \ast \\[2pt]
	\ast\ast & c(-b+am\tau) + ad
\end{pmatrix}.
\end{eqnarray*}
The trace $\mathrm{tr}\rho(x^{-n} y x^m y^{-1})$ is $2 + c^2 m(m-1)\tau^2$.

\begin{claim}
\label{Margulis}
$c \ne 0$
\end{claim}

\begin{proof}[Proof of Claim~\ref{Margulis}]
Suppose for a contradiction that $c = 0$. 
If $\rho(y) = 
\begin{pmatrix}
	1 & \tau' \\[2pt]
	0 & 1
\end{pmatrix}$, 
then $[\rho(x), \rho(y)] = I$, and hence $[x, y] = 1$, contradicting the assumption. 
So 
$\rho(y) = 
\begin{pmatrix}
	a & b \\[2pt]
	0 & a^{-1}
\end{pmatrix}\ (a \ne \pm 1)$. 
Note that both $\rho(x)$ and $\rho(y)$ fix $\infty$ in the sphere at infinity $S^2_{\infty} = \hat{\mathbb{C}}$. 
The Margulis lemma \cite{Rat} says that for sufficiently small $\varepsilon > 0$, 
the $\varepsilon$-thin part of $\mathrm{int}(X)$ is a horo-torus, 
which is a quotient of a sufficiently small horo-ball $B$ by $\Gamma$ generated by two parabolic isometries fixing $\infty \in S^2_{\infty} = \hat{\mathbb{C}}$. 
If we have a loxodromic element $\rho(y)$ which fixes $\infty \in S^2_{\infty} = \hat{\mathbb{C}}$, 
then $\rho(y)^p (B) \subset B$ for positive integer $p$ or for negative integer $p$. 
This means that $B/\Gamma$ does not embed in $M$. 
Hence any loxodromic isometry in $\rho(G(K))$ does not fix $\infty$ in the sphere at infinity $S^2_{\infty} = \hat{\mathbb{C}}$. 
In particular, $c \ne 0$. 
($\rho(y)$ may be a parabolic element which does not fix $\infty$.)
\end{proof}

Therefore for non-zero constants $c, \tau$, 
the value $2 + c^2 m(m-1) \tau^2$ can be $\pm 2$ for at most one integer $m > 1$. 
Furthermore, for integers $m, n >1$, 
we have $2 + c^2 m(m-1) \tau^2 = 2 + c^2 n(n-1) \tau^2$ if and only if $m = n$. 
Since the trace is invariant under conjugation, 
$x^{1-m} y x^m y^{-1}$ and  $x^{1-n} y x^n y^{-1}$ $(m, n >1)$ are not conjugate whenever $m \ne n$, 
and $x^{1-m} y x^m y^{-1}$ can be peripheral for at most one integer $m > 1$. 
\end{proof}

\begin{remark}
\label{peripheral_persistent}
Let $g$ be a persistent element in $G(K)$. 
If $g$ is peripheral, 
then it is conjugate to a power of a slope element $\gamma$ which corresponds to $r \in \mathbb{Q} \cup \{ \infty \}$.  
Obviously $g \in \langle\!\langle r \rangle\!\rangle$. 
Since $g$ is persistent, $r$ should be $\infty$. 
So a peripheral persistent element is conjugate to a power of the meridian of $K$. 
\end{remark}

\section{Minimizers for subgroups}

The goal of this section is to prove Theorem~\ref{minimizer}, which asserts that every subgroup of the knot group of a torsion-free hyperbolic knot admits a minimizer.
We first give the proof in Subsection~\ref{minimizers} assuming Lemma~\ref{S_K_intersection_no_torsion}, and then prove Lemma~\ref{S_K_intersection_no_torsion} in Subsection~\ref{shrink}.

\subsection{Existence of minimizers}
\label{minimizers}

Let $H$ be a nontrivial subgroup of $G(K)$. 
By definition, for any element $h \in H$, 
$\mathcal{S}_{\cap}(H) \subset \mathcal{S}(h)$.  
Recall that an element $h_0$ is called a minimizer of $H$ if $\mathcal{S}(h_0) = \mathcal{S}_{\cap}(H)$ holds. 

The goal in this section is to prove that a minimizer always exists. 

\begin{thm_minimizer}
Let $K$ be a torsion-free hyperbolic knot. 
Then every nontrivial subgroup $H$ of $G(K)$ has a minimizer $h_0 \in H$, 
i.e. an element $h_0$ that satisfies $\mathcal{S}(h_0) = \mathcal{S}_{\cap}(H)$.
\end{thm_minimizer}

\begin{proof}
We consider two cases according as $H$ is conjugate into the peripheral subgroup $P(K) = i_*(\pi_1(\partial E(K)))$ or not. 

\medskip

\noindent
\textbf{Case 1.}\ 
Assume that  $H$ is conjugate into the peripheral subgroup of  $P(K) = i_*(\pi_1(\partial E(K))) \cong \mathbb{Z} \oplus \mathbb{Z}$. 
By taking some conjugation in $G(K)$, 
we may assume $H \subset P(K) \cong \mathbb{Z} \oplus \mathbb{Z}$. 

Suppose first that $H$ is cyclic. 
Then $H = \langle \gamma^k \rangle$ for some slope element $\gamma = \mu^p\lambda^q$ and non-zero integer $k$. 
In the case where $\gamma = \mu$, $H = \langle \mu^k \rangle$, 
and $\mathcal{S}_{\cap}(H) =  \mathcal{S}(\mu^k) = \emptyset$.  
Assume that $\gamma$ is not a meridian. 
Then, since $K$ has no finite surgery slope by the assumption, 
\cite[Proposition~3.1]{IMT_realization} implies that $\mathcal{S}_{\cap}(H) = \mathcal{S}(\gamma^k) = \{ p/q \}$. 

Next suppose that $H$ is non-cyclic. 
We may write  
$H = \langle (\mu^a \lambda^b)^m,\  (\mu^c \lambda^d)^n  \rangle$, where $a$ and $b$ are coprime, 
and $c$ and $d$ are coprime, 
and $m, n \ne 0$. 
Then $\left(\mu^a \lambda^b)^m\right)^{dn} \left((\mu^c \lambda^d)^n\right)^{-bm} = \mu^{(ad-bc)mn} \in H$. 
Since $\mathcal{S}(\mu^{(ad-bc)mn}) = \emptyset$ (because $K$ has no torsion surgery), 
obviously $\mathcal{S}_{\cap}(H) = \emptyset$ as well. 

\medskip

\noindent
\textbf{Case 2.}
Assume that $H$ is not conjugate into the peripheral subgroup of $P(K) = i_*(\pi_1(\partial E(K)))$. 

\medskip

We observe:

\begin{claim}
\label{peripheral_subgroup_element}
$H$ contains a non-peripheral element\footnote{Since $H$ is not conjugate into the subgroup $P(K)$, 
there is no element $\gamma \in G(K)$ such that 
$\gamma^{-1} H \gamma \subset P(K)$. 
However, for each element $h \in H$, 
there may be $\gamma_h \in G(K)$ such that 
$\gamma_h^{-1} h \gamma_h \in P(K)$, 
i.e. $h$ is peripheral.  
Claim~\ref{g_1_non-peripheral} excludes such a possibility.}.
\end{claim}

\begin{proof}
If $H$ is cyclic, by the assumption each nontrivial element is non-peripheral.  
(Actually if $H = \langle g \rangle$ and $g$ is non-peripheral, then its image of holonomy representation is not parabolic and 
so is $g^n$ for all $n \ne 0$; see the proof of Lemma~\ref{non-conjugate}.)
Assume that $H$ is non-cyclic. 
Since $K$ is hyperbolic and $H$ is not conjugate into the peripheral subgroup, 
$H$ is non-abelian. 
Hence we have $x, y$ with $[x, y] \ne 1$. 
If $x$ is non-peripheral, then $x$ is a desired non-peripheral element in $H$.
So we may assume both $x$ and $y$ are peripheral. 
Then $x^{1-m} y x^m y^{-1}$ are non-peripheral for all but one integer $m > 1$ (Lemma~\ref{non-conjugate}). 
\end{proof}

\begin{claim}
\label{g_1_non-peripheral}
There exists a non-peripheral element $g_1 \in H$ such that 
\begin{equation} 
\label{eqn:condition}
 \langle\, p_{\mathrm{abel}}( g_1)\,  \rangle = p_{\mathrm{abel}}(H) \subset H_1(E(K);\mathbb{Z}) \cong \mathbb{Z}, 
\end{equation} where $p_{\mathrm{abel}}:G(K) \rightarrow G\slash[G,G] = H_1(E(K); \mathbb{Z})$ is the abelianization map.
\end{claim}

\begin{proof}
If $p_{\mathrm{abel}}(H) = \{ 0 \}$, then any element $g_1 \in H$ satisfies $p_{\mathrm{abel}}( g_1) = 0$, 
and $\langle\, p_{\mathrm{abel}}( g_1)\,  \rangle = p_{\mathrm{abel}}(H)$.  
So in the following we suppose $p_{\mathrm{abel}}(H) \ne \{ 0 \}$. 

Let $h$ be a nontrivial element of $H$ such that $p_{\mathrm{abel}}(h) =[h]$ generates $p_{\mathrm{abel}}(H) \subset \mathbb{Z}$, 
i.e. $\langle [h] \rangle = p_{\mathrm{abel}}(H)$. 
If $h$ is non-peripheral, then we put $g_1 = h$. 
Assume that $h$ is peripheral. 
Apply Claim~\ref{peripheral_subgroup_element} to take a non-peripheral element $g \in H$.  
Then $[g, h] \ne 1$, and we take $h^{1-m} g h^m g^{-1}$, 
which is non-peripheral for some integer $m > 1$; see Lemma~\ref{non-conjugate}. 
(Actually there are infinitely many such integers $m$.) 
Set $g_1 = h^{1-m} g h^m g^{-1}$, which satisfies $[g_1] = [h^{1-m} g h^m g^{-1}] = [h] \in p_{\mathrm{abel}}(H)$. 
\end{proof}

Assume that $\mathcal{S}_{\cap}(H) = \{ s_1, \dots, s_n \}$; 
if $n = 0$, then $\mathcal{S}_{\cap}(H)$ is understood to be the empty set. 
To find a minimizer of $H$, 
we first choose finitely many elements $g_1, \dots, g_\ell \in H$ so that $\mathcal{S}(g_1) \cap \cdots \cap \mathcal{S}(g_\ell) =  \{ s_1, \dots, s_n \}$ and they satisfy 
additional conditions as in the next claim.

\begin{claim}
\label{finite_intersection_realize}
There exist elements $g_1,\ldots, g_\ell \in H $ such that 
\begin{itemize}
\item[(i)] $\mathcal{S}(g_1) \cap\mathcal{S}(g_2)\cap  \mathcal{S}(g_3)\cap \cdots \cap \mathcal{S}(g_\ell) =  \{ s_1, \dots, s_n \}$. 
\item[(ii)] $g_1$ is non-peripheral, and $g_2,\ldots, g_{\ell} \in H \cap [G(K), G(K)]$.
\end{itemize}
\end{claim}

\begin{proof}
Take a non-peripheral element $g_1$ of $H$ as in Claim~\ref{g_1_non-peripheral}. 

Since $K$ is hyperbolic, 
$\mathcal{S}(g_1)$ is a finite subset of $\mathbb{Q}$; see \cite{GM,Osin,IchiMT}. 
So we put
$\mathcal{S}(g_1) =\{s_1,\ldots, s_{m}\}$ $(m\geq n)$. 
If $m = n$, then $g_1$ is a desired element. 

Assume that $m > n$. 
Then since 
$\mathcal{S}_{\cap}(H) \not\ni s_m$, 
there exists an element $g'_2 \in H$ such that $p_{s_m}(g'_2) \ne 1$. 

If $g'_2 \in [G(K), G(K)]$ we just take $g_2=g'_2$, which satisfies $p_{s_m}(g_2) = p_{s_m}(g'_2) \ne 1$. 
If $g'_2 \not \in [G(K),G(K)]$, 
then by our choice (\ref{eqn:condition}) of $g_1$, 
there exists a non-zero integer $k$ such that 
$[g_1^k g'_2] = k[g_1] + [g'_2] = 0 \in H_1(E(K); \mathbb{Z}) \cong \mathbb{Z}$, 
i.e. $g_1^{k}g'_2  \in [G(K),G(K)]$. 
In this case we take $g_2=g_1^{k}g'_2 \in H \cap [G(K), G(K)]$. 
(Note that $\{ s_1, \dots, s_n \} = \mathcal{S}_{\cap}(H) \subset S(g_2)$.)
Then $p_{s_m}(g_2) = p_{s_m}(g_1^{k}g'_2)=p_{s_m}(g'_2) \neq 1$. 
In either case, 
$s_m \not\in \mathcal{S}(g_2)$, 
in particular, $\{ s_1, \dots, s_n \}  \subset \mathcal{S}(g_1) \cap \mathcal{S}(g_2)$, but 
$s_m \not\in \mathcal{S}(g_1) \cap \mathcal{S}(g_2)$. 
Hence, 
\[
\{ s_1, \dots, s_n \} \subset \mathcal{S}(g_1) \cap \mathcal{S}(g_2) \subset \{s_1,\ldots, s_{m-1}\}.
\]  
By repeating the same argument (at most $m - n$ times) we get desired elements $g_1,\ldots,g_{\ell} \in H \cap [G(K), G(K)]$ 
satisfying 
\[
\mathcal{S}(g_1) \cap\mathcal{S}(g_2)\cap  \mathcal{S}(g_3)\cap \cdots \cap \mathcal{S}(g_\ell) = \{ s_1, \dots, s_n \}.
\] 
This finishes the proof of Claim~\ref{finite_intersection_realize}. 
\end{proof}

Now we proceed to find a single element $h_0 \in H$ 
satisfying 
\[
\mathcal{S}(h_0) = \{ s_1, \dots, s_n \}.
\]
In the above, if $\ell = 1$, then we just take an element $g_1$ as $h_0$. 

In what follows, we suppose that $\ell > 1$, 
and apply Lemma~\ref{S_K_intersection_no_torsion} in Subsection~\ref{shrink} to the non-peripheral element $g_1 \in H$ and the nontrivial element 
$g_2 \in H \cap [G(K), G(K)]$ 
to obtain a non-peripheral element $g_1^{m_1} g_2 \in H$ that satisfies $ \mathcal{S}(g_1^{m_1} g_2) = \mathcal{S}(g_1) \cap\mathcal{S}(g_2)$. 

So we have 
\[ \mathcal{S}(g_1^{m_1} g_2) \cap \mathcal{S}(g_3) \cap \cdots \cap \mathcal{S}(g_\ell) = \{ s_1, \dots, s_n \},
\]
where $g_1^{m_1} g_2 \in H$ is non-peripheral and $g_3, \dots, g_{\ell} \in H \cap [G(K), G(K)]$. 
This means that the replacement of the pair $g_1, g_2$ with a single 
non-peripheral element $g_1^{m_1} g_2$ decreases $\ell$, 
the cardinality of the set $\{g_1, \ldots,, g_{\ell}\}$. 
 
Repeating the same argument, for suitably taken $m_2,m_3,\ldots, m_{\ell-1}$ we conclude that 
\[ h_0 = (\cdots ((g_1^{m_1} g_2)^{m_2} g_3)^{m_3} \cdots g_{\ell-1})^{m_{\ell-1}}g_{\ell}\]
is a non-peripheral element $h_0 \in H$ which satisfies
\[
\mathcal{S}(h_0) = \{ s_1, \dots, s_n \}. 
\]

\medskip
This completes a proof of Theorem~\ref{minimizer}. 
\end{proof}

\medskip

\subsection{Shrinking Lemma}
\label{shrink}

Let $g, h$ be elements of $G(K)$. 
In general $\mathcal{S}(g) \cap \mathcal{S}(h) \subset \mathcal{S}(gh)$ and 
$\mathcal{S}(gh)$ may become quite large. 
However, under some condition we have the following. 

\begin{lemma}[Shrinking Lemma]
\label{S_K_intersection_no_torsion}
Let $K$ be a torsion-free hyperbolic knot. 
Let $g$ be a non-peripheral element and $h$ a nontrivial element in $[G(K),G(K)]$. 

Then we have a constant $N > 0$ such that  $g^m h$ is non-peripheral, and 
\[
\mathcal{S}(g) \cap \mathcal{S}(h) = \mathcal{S}(g^m h) 
\] 
for all integers $m \ge N$. 
\end{lemma}

In \cite[Lemma 1.16]{IMT_realization}, 
we have shown the lemma in more general setting, 
but with slightly different assumption that essentially says that $K$ does not have finite surgery
slope nor two reducing surgery slopes.

Although Lemma 3.4 is proved in a similar manner as in the proof of \cite[Lemma 1.16]{IMT_realization}, 
since we use essentially this lemma in the proof of Theorem 1.7, 
for completeness, we give its proof here. 

\medskip

\begin{proof}
Obviously we have the following inequality. 
\[
\mathcal{S}(g) \cap \mathcal{S}(h) \subset  \mathcal{S}(g^mh) 
\]
for all integers $m$. 

In the following we will prove that this inequality may be replaced by the equality under the condition in Lemma~\ref{S_K_intersection_no_torsion}.  

It is sufficient to prove that there exists a constant $N$ such that 
\[
\mathcal{S}(g) \cap \mathcal{S}(h) \supset \mathcal{S}(g^m h)
\]
for any integer $m \ge N$.

Indeed we show that if $s \not\in \mathcal{S}(g)$ or $s \not\in \mathcal{S}(h)$, 
then $s \not\in \mathcal{S}(g^mh)$. 
We divide the argument into two cases depending upon $s \in \mathcal{S}(g) - \mathcal{S}(h)$ or $s \not\in \mathcal{S}(g)$. 

\medskip

\noindent
\textbf{Case 1.}\ $s \in \mathcal{S}(g) - \mathcal{S}(h)$. 

Since $s \in \mathcal{S}(g) - \mathcal{S}(h)$, 
we have $p_s(g^{m}h) = p_s(g)^m p_s(h) = p_s(h) \ne 1$, and thus $s \not\in \mathcal{S}(g^mh)$ for all integres $m$.

\medskip

\noindent
\textbf{Case 2.}\ $s \not \in \mathcal{S}(g)$. 

We will prove that there exists a constant $N$ such that $s \not\in \mathcal{S}(g^mh)$ for all integers $m \ge N$.  
In what follows we will find such constants $N_1$ and $N_2$ corresponding to the cases where 
$s$ is a non-hyperbolic and hyperbolic surgery slope, respectively.
Then put $N = \mathrm{max}\{ N_1, N_2 \}$. 

\medskip

\noindent
\textbf{(1) Constant $N_1$ for non-hyperbolic surgery slopes.} 

We first note that there is at most one integer $m$ such that $p_s(g^mh) = 1$. 
In fact, if we have two distinct such integers $m_1$ and $m_2$, 
then $p_s(g^{m_1}) = p_s(h)^{-1} = p_s(g^{m_2})$, 
which means $p_s(g)^{m_1 - m_2} = 1$. 
Since $p_s(g) \ne 1$ and $\pi_1(K(s))$ is torsion-free, 
$m_1= m_2$, a contradiction. 

Since there are at most finitely many non-hyperbolic surgeries, and for each such a slope $s$, 
$p_s(g^mh) = 1$ for at most one integer $m$. 
So we may take a constant $N_1 > 0$ so that $p_s(g^mh) \ne 1$ provided $m \ge N_1$ for all non-hyperbolic surgeries. 

\medskip

\noindent
\textbf{(2) Constant $N_2$ for hyperbolic surgery slopes.} 

To find the constant $N_2$, 
we briefly recall stable commutator length which will be needed in our discussion. 

For $a \in [G,G]$ the \emph{commutator length\/} $\mathrm{cl}_G(a)$ is the smallest number of commutators in $G$ whose product is equal to $a$. 
The \emph{stable commutator length\/} $\mathrm{scl}_{G}(a)$ of $a \in [G,G]$ is defined to be the limit
\begin{equation}
\label{def1}
\mathrm{scl}_{G}(a) = \lim_{n \to \infty} \frac{\mathrm{cl}(a^n)}{n}.
\end{equation}

Since $\mathrm{cl}_G(a^n)$ is non-negative and subadditive, Fekete's subadditivity lemma shows that 
this limit exists.  

We will extend (\ref{def1}) to the stable commutator length $\mathrm{scl}(a)$ for an element $a$ which is not necessarily in $[G, G]$ as 
\begin{equation}
\label{def2}
\mathrm{scl}_G(a) = \begin{cases}
\frac{\textrm{scl}_G(a^{k})}{k} & \mbox{ if } a^{k} \in [G,G] \mbox{ for some } k > 0,\\
\infty & \mbox{otherwise}.
\end{cases}
\end{equation}

Then we have the following property; for the proof see \cite[Lemma~2.1]{IMT_decomposition} for instance. 

\begin{claim}
\label{scl_a^k}
For any $a \in G$ and $k > 0$,  
we have 
\begin{equation}
\label{power}
\mathrm{scl}_G(a^{k})=k\,\mathrm{scl}_G(a).
\end{equation} 
\end{claim}

Since a homomorphism sends a commutator to a commutator, 
we observe the following monotonicity property of scl \cite[Lemma~2.4]{Cal_MSJ}. 

\begin{claim}
\label{monotonicity}
Let $\varphi\colon G \to H$ be a homomorphism. 
Then 
\begin{equation}
\label{mono}
\mathrm{scl}_H(\varphi(a)) \leq \mathrm{scl}_G(a)\quad  \textrm{for all}\quad a \in G.
\end{equation}
\end{claim}

The following inequality  may be well-known for experts, 
which is proved by noting that 
$(ab)^{2n}a^{-2n}b^{-2n}$ can be written as a product of $n$ commutators \cite[p.45]{Cal_MSJ}; 
see \cite[Lemma~4.6]{IMT_realization} for an alternate proof. 

\begin{claim}
\label{scl_product}
Let $G$ be a group. 
Assume that $a, b$ \(and hence $ab$\) belong to $[G, G]$ or 
the abelianization $G/[G, G]$ is a finite group\footnote{In practice, we apply this lemma for $G = \pi_1(K(s))$ with non-zero slope $s \in \mathbb{Q}$.}. 
Then we have 
\begin{equation}
\label{product}
\mathrm{scl}_G(ab) \ge \mathrm{scl}_G(a) -\mathrm{scl}_G(b) -\frac{1}{2}.
\end{equation}
\end{claim}

\medskip

Since $g$ is non-peripheral, 
Thurston's hyperbolic Dehn surgery theorem \cite{T1,T2} and Calegari's length inequality \cite{Cal_GAFA} give us a constant $\delta_g > 0$ (depending only on $g$) so that for all hyperbolic surgery slopes $s$, 
\begin{equation}
\label{delta}
\mathrm{scl}_{\pi_1(K(s))}(p_s(g)) > \delta_g > 0\ \textrm{whenever}\ p_s(g) \ne 1,\ \textrm{i.e.}\  s \not\in \mathcal{S}(g).
\end{equation}
See \cite[Theorem~4.8]{IMT_realization} for details. 
Furthermore, 
since $h \in [G(K), G(K)]$, 
its stable commutator length satisfies $\mathrm{scl}_{G(K)}(h) < \infty$.  

Thus we set a constant $C$ as 
\[
C = \frac{\mathrm{scl}_{G(K)}(h)+\frac{1}{2}}{\delta_g} > 0. 
\]

Then for any integer $m \ge C$, 
we have 
\begin{equation}
\label{m_delta}
m\, \delta_g \ge \mathrm{scl}_{G(K)}(h)+\frac{1}{2}. 
\end{equation}

Assume that $s \ne 0$. 
Then $G = \pi_1(K(s))$ satisfies that $G/[G, G]$ is finite. 
Take an integer $m$ so that $m \ge C$.  
Then  we have: 
\begin{align*}
\mathrm{scl}_{\pi_1(K(s))}(p_s(g^{m}h)) &\ge m\, \mathrm{scl}_{\pi_1(K(s))}(p_s(g)) - \mathrm{scl}_{\pi_1(K(s))}(p_s(h))\,-\,\frac{1}{2}\quad {\small(\ref{product}\ \textrm{and}\  \ref{power})}\\
& > m\, \delta_{g} - \mathrm{scl}_{G(K)}(h) - \frac{1}{2}\quad {\small(\ref{delta},\ \ref{mono})}\\
&  \ge 0 \quad {\small(\ref{m_delta})}. 
\end{align*}
Hence, $\mathrm{scl}_{\pi_1(K(s))}(p_s(g^{m}h)) > 0$ and $p_s(g^{m}h)\ne 1$. 

\medskip

Let us suppose that $s= 0$ and $p_0(g) \not\in [\pi_1(K(0)), \pi_1(K(0))]$. 
(This is the only case we cannot apply Lemma~\ref{scl_product}.)
Then we have: 

\begin{claim}
\label{s=0}
$p_0(g^mh) \ne 1$ for all $m > 0$. 
\end{claim}

\begin{proof}
Suppose that $p_0(g^mh) = 1$, i.e. $p_0(g)^m = p_0(h)^{-1}$. 
Abelianizing this we have $m [p_0(g)] = - [p_0(h)] \in H_1(K(0)) = \mathbb{Z}$. 
Since $h \in [G(K), G(K)]$, 
$p_0(h) \in [\pi_1(K(0)), \pi_1(K(0))]$ and 
we have $[p_0(h)] = 0 \in H_1(K(0))$. 
On the other hand, 
since $p_0(g) \not\in [\pi_1(K(0)), \pi_1(K(0))]$, 
$[p_0(g)] \ne 0$. 
Hence, $m = 0$.  
\end{proof}

Put $N_2 = C > 0$. 
Then the above discussion shows that $p_s(g^{m}h) \ne 1$ for all $m \ge N_2$, i.e. 
$s \not\in \mathcal{S}(g^mh)$ for all integers $m \ge N_2$. 

Set $N_3 = \mathrm{max}\{ N_1, N_2 \}$, 
which is the desired constant in Case 2. 

\medskip

Combining Cases 1 and 2, 
we have $\mathcal{S}(g^n h) = \mathcal{S}(g) \cap \mathcal{S}(h)$ for all $n \ge N_3$. 

\medskip

To finish the proof of Lemma~\ref{S_K_intersection_no_torsion}, 
we show that we may take $g^n h$ to be non-peripheral.  

\begin{lemma}
\label{non-peripheral_constant}
Let $K$ be a hyperbolic knot. 
For any non-peripheral element $\alpha$ and any element $\beta$ of $G(K)$, 
there exists a constant $C$ \(depending on $\alpha$ and $\beta$\) such that $\alpha^n \beta$ is non-peripheral whenever $n \ge C$. 
\end{lemma}

\begin{proof}
Consider a holonomy representation $\rho \colon G(K) \to  PSL_2(\mathbb{C})$ such that 
$\rho(\alpha) = 
\begin{pmatrix}
	a & 0\\[2pt]
	0 & a^{-1}
\end{pmatrix}$
for some $1 \ne a \in \mathbb{C}$. 
Since $\alpha$ is not a torsion element, 
$\rho(\alpha)$ is not elliptic, and hence $|| a || \ne 1$; we may assume $|| a || > 1$. 
Write  
$\rho(\beta) = 
\begin{pmatrix}
	x & y\\[2pt]
	z & u
\end{pmatrix}$.  
Then the trace  $\mathrm{Tr}\rho(\alpha^m \beta) = \mathrm{Tr}\rho(\alpha)^m \rho(\beta)$ is  $a^m x + a^{-m} u$. 

Let us divide into two cases depending upon whether $(x, u) = (0, 0)$ or not. 
If $(x, u) = (0, 0)$, then $\mathrm{Tr}\rho(\alpha^m \beta) = 0 \ne \pm 2$. 
Hence $\alpha^m \beta$ is not parabolic, and thus it is non-peripheral. 

Assume that $(x, u) \ne (0, 0)$. 
Since $||a||^m  |x|  - ||a^{-1}||^m  |u| \le || a^m x + a^{-m} u || \le ||a||^m  |x|  + ||a^{-1}||^m  |u| $ and 
$|| a ||^m \to \infty$ and $|| a^{-1} ||^m \to 0$ when $m \to \infty$, 
we have a constant $C > 0$ such that 
$2 < || a^m x + a^{-m} u ||$ if $x \ne 0$ and $|| a^m x + a^{-m} u || < 2$ if $x = 0$ ($u \ne 0$) 
provided $m \ge C$. 
This means that $\alpha^{m} \beta$ is non-peripheral if $m \ge C$. 
\end{proof}

Finally set $N = \mathrm{max}\{ N_3, C \}$. 
Then
\[
\mathcal{S}(g) \cap \mathcal{S}(h) = \mathcal{S}(g^m h) 
\] 
for all integers $m \ge N$ and $g^m h$ is non-peripheral. 

This completes a proof of Lemma~\ref{S_K_intersection_no_torsion}. 
\end{proof}

\medskip

\section{Characterization of subgroups containing persistent elements}

In this section we prove Theorems~\ref{characterization}, 
which gives a complete characterization of subgroups which contain persistent elements. 

\begin{thm_characterization}
Let $K$ be a torsion-free hyperbolic knot, 
and $H$ a subgroup of $G(K)$. 
Then the following three conditions are equivalent. 

\begin{enumerate}
\item 
$H$ contains a persistent element. 
\item
$H$ is not contained in $\langle\!\langle r \rangle\!\rangle$ for any $r \in \mathbb{Q}$.  
\item 
$\mathcal{S}_{\cap}(H) = \emptyset$.
\end{enumerate}
\end{thm_characterization}

\begin{proof}
\noindent
$(1) \Rightarrow (2)$.\ 
Assume that $H \subset \langle\!\langle r \rangle\!\rangle$ for some $r \in \mathbb{Q}$. 
Then obviously every element in $H$ becomes trivial for $r$--Dehn filling, and cannot be persistent. 

\medskip

\noindent
$(2) \Rightarrow (3)$.\  
Assume that $r \in \mathcal{S}_{\cap}(H)$. 
Then, by definition, for every $h \in H$ we have $r \in \mathcal{S}(h)$, 
i.e. $p_r(h) = 1$, and hence $h \in \langle\!\langle r \rangle\!\rangle$. 
This shows that $H \subset \langle\!\langle r \rangle\!\rangle$. 

\medskip

\noindent
$(3) \Rightarrow (1)$.\ 
If $\displaystyle
\mathcal{S}_{\cap}(H) = \emptyset$, 
then by Theorem~\ref{minimizer}, 
we have an element $h_0 \in H$ with 
$ \mathcal{S}(h_0) =  \displaystyle\mathcal{S}_{\cap}(H) =\emptyset$, 
which means that $h_0$ is a persistent element. 
\end{proof}

\medskip

\begin{remark}
Assume that $H$ is not conjugate into $P(K) = i_*(\partial E(K))$.  
Then by the proof of Theorem~\ref{minimizer}, $H$ contains a non-peripheral, persistent element $h_0$.  
Furthermore, for any $x \in H$, $h_x = h_0^{n-m} x h_0^m x^{-1} \in H$ is persistent. 
Let us observe that  $h_x = h_y$ if and only if $x = y$. 
Suppose that 
$h_0^{n-m} x h_0^m x^{-1} = h_0^{n-m}  y h_0^m y^{-1}$ for some distinct elements $x, y \in H$. 
This implies $x h_0^m x^{-1} =  y h_0^m y^{-1}$, i.e. 
$y^{-1} x h_0^m  =  h_0^m y^{-1}x$. 
This shows $h_0^m$ and $y^{-1} x$ commute. 
Since $K$ is hyperbolic, 
$h_0^m$, and hence $h_0$ is peripheral, contradicting the assumption. 
Thus varying $x \in H$, we have infinitely many persistent element in $H$ by Theorem~\ref{persistent_BS}. 
\end{remark}

\medskip

\section{Complements of subgroups and persistent elements}

Theorem~\ref{characterization} shows a necessary condition for a subgroup to have persistent elements is also a sufficient condition. 
In this section, let us turn to find persistent elements in the complement of a given proper subgroup. 

\begin{thm_complement}
Let $K$ be a torsion-free hyperbolic knot. 
Then for any proper subgroup $N$ of $G(K)$, 
there exists a persistent element $g$ in $G(K) - N$.
\end{thm_complement}

We first observe the following. 

\begin{lemma}
\label{ubiquitous_non-peripheral}
Let $K$ be a hyperbolic knot. 
For any proper subgroup $N \subset G(K)$, 
there exists a non-peripheral element $y$ in $G(K) - N$. 
\end{lemma}

\begin{proof}
If $N$ consists of peripheral elements, 
then any non-peripheral element does not belong to $N$, and $G(K) - N$ has a non-peripheral element. 
Assume that $N$ contains a non-peripheral element. 
Take a non-peripheral element $\alpha \in N$ and a nontrivial element $x \not\in N$. 
Then Lemma~\ref{non-peripheral_constant} shows that 
$\alpha^n x$ is a non-peripheral element for infinitely many integers $n$. 
Let us take a non-peripheral element $g = \alpha^n x$. 
Now we see that it belongs to $G(K) - N$. 
If $g = \alpha^n x \in N$, 
then, since $\alpha^n \in N$, $x = \alpha^{-n} g \in N$, contradicting the choice of $x$. 
\end{proof}

In Lemma~\ref{S_K_intersection_no_torsion}, 
we take a persistent element $h \in [G(K),G(K)]$ (i.e.  $\mathcal{S}(h) = \emptyset$) so that 
$\mathcal{S}(g^m h) = \mathcal{S}(g) \cap \mathcal{S}(h) = \emptyset$. 
Then we have: 

\begin{lemma}
\label{g^mh}
Let $K$ be a torsion-free hyperbolic knot. 
For any persistent element $h \in [G(K), G(K)]$ and any non-peripheral element $g \in G(K)$, 
there is a constant $C > 0$ such that $g^m h$ is a persistent element for $m \ge C$. 
\end{lemma}

\begin{proof}[Proof of Theorem~\ref{complement}]
Following Example~\ref{Seifert_surface} we may take a persistent element $h$ in the fundamental group of an incompressible Seifert surface. 
In particular, $h \in [G(K), G(K)]$. 
If $h \not\in N$, then $h$ is a desired persistent element. 
Assume that $h$ belong to $N$. 
Then take a non-peripheral element $g \not\in N$ (Lemma~\ref{ubiquitous_non-peripheral}) and 
apply Lemma~\ref{g^mh} to see that $g^m h$ is a persistent element whenever $m > C$ for some constant $C$.
Then at least one of $g^m h$ or $g^{m+1} h$ is a persistent element which does not belong to $N$. 
If both $g^m h$ and $g^{m+1}$ belong to $N$,  
then we get $g = (g^{m+1} h) (g^m h)^{-1} \in N$. 
This contradicts our choice of $g$.  
\end{proof}

\medskip

\section{Pseudo-meridians and persistent elements}
\label{p-m}

In this section, we investigate a relationship between persistent elements and pseudo-meridians. 
Since any pseudo-meridian normally generates the knot group $G(K)$, 
it also generates $H_1(E(K))$. 
Thus, if a persistent element is homologically trivial, 
it cannot be a pseudo-meridian. 
What can we say about the converse? 
By definition, any automorphic image of a pseudo-meridian is a pseudo-meridian. 
In contrast, the proposition below establishes the existence of a persistent element equivalent to a non-persistent element.

\begin{proposition}
\label{persistent_auto}
Let $K=K_1\#K_2 $ be a non-prime knot with a negative amphicheiral knot $K_1$. 
Then there exists a persistent element which is equivalent to a non-persistent element. 
\end{proposition}

\begin{proof}
Following \cite{Tsau2} (attributed to Simons), 
we begin by constructing an automorphism $\varphi \colon G(K) \rightarrow G(K)$ which is not induced by a homeomorphism of $E(K)$. 

Let $\mu$ and $\mu_i$ denote the meridians of $K$ and $K_i$, respectively. 
Similarly, let $\lambda$ and $\lambda_i$ denote the preferred longitudes of $K$ and $K_i$, respectively. 
Then the meridian satisfies $\mu = \mu_1=\mu_2$, and the longitude $\lambda$ can be written as $\lambda_1 \lambda_2$; 
see Figure~\ref{composing_persistent}. 

Since $K_1$ is a negative amphicheiral knot, 
$G(K_1)$ admits an automorphism $\varphi_1 \colon G(K_1) \rightarrow G(K_1)$ such that 
$\varphi(\lambda_1)=\lambda_1^{-1}, \varphi(\mu_1) = \mu_1$. 
Since $G(K) = G(K_1) \ast_{\mu_1 = \mu_2} G(K_{2})$, $\varphi_1 \ast id$ defines an automorphism 
$\varphi \colon G(K) \rightarrow G(K)$ such that $\varphi(\lambda) = \varphi(\lambda_1 \lambda_2)= \lambda_1^{-1} \lambda_2$, 
$\varphi(\mu)=\mu$. 

For every integer $p \in \mathbb{Z}$, 
we have 
\[
\varphi(\lambda_1^{-1}\lambda_2\mu^p) = (\lambda_1 \lambda_2) \mu^{p} = \lambda \mu^{p}, 
\] 
which is a slope element of $G(K)$, 
i.e. $\lambda_1^{-1}\lambda_2\mu^p$ is equivalent to the slope element $\lambda \mu^{p}$.  

By definition, 
$p \in \mathcal{S}(\lambda \mu^{p})$.
In what follows, we show that $\mathcal{S}(\lambda_1^{-1} \lambda_2\mu^p) = \emptyset$, 
meaning that $\lambda_1^{-1}\lambda_2\mu^p$ is persistent. 

Now let us consider a decomposition $E(K) = E(K_1) \cup E(K_2) \cup X$, 
where $X$ is the two-fold composing space $[\textrm{disk with two holes}]\, \times S^1$. 
Note that a regular fiber of $X$ is a meridian $\mu\, (=\mu_1 = \mu_2)$.

\begin{figure}[htb]
\centering
\includegraphics[width=0.5\textwidth]{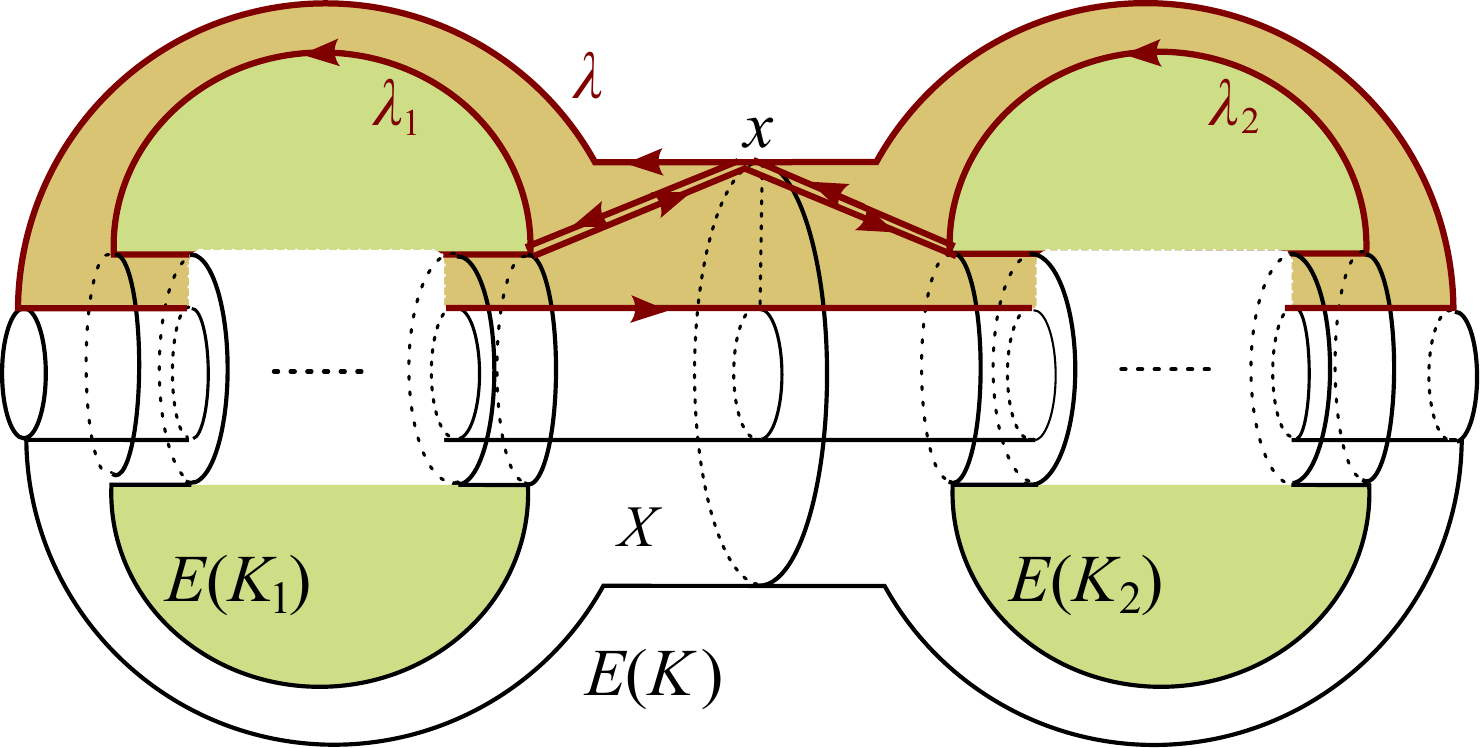}
\caption{Longitudes $\lambda_1, \lambda_2$, and $\lambda = \lambda_1 \lambda_2$}
\label{composing_persistent}
\end{figure}

Then $K(r)$ is expressed as $E(K_1) \cup E(K_2) \cup \left(X \cup_r (S^1 \times D^2)\right)$, 
and $X \cup_r (S^1 \times D^2)$ is a Seifert fiber space over the annulus with at most one exceptional fiber. 
In particular it is boundary-irreducible for all $r \in \mathbb{Q}$. 
Let us show that $\lambda_1^{-1}\lambda_2\mu^p$ survives in $\pi_1(X \cup_r (S^1 \times D^2))$ for all $r \in \mathbb{Q}$. 
Recall that $\pi_1(X \cup_r (S^1 \times D^2))$ has a presentation
\[
\langle c, d, t \mid [c, t] = [d, t] = 1, c^{\alpha} = t^{\beta} \rangle, 
\]
where $\mu = t,\ \lambda_1= d$ and $\lambda_2 = d^{-1}c$. 
Putting $t = c = 1$, we obtain 
$\langle d \mid \ \rangle$, 
in which $\lambda_1^{-1}\lambda_2\mu^p = d^{-2} \ne 1$. 
Thus $\lambda_1^{-1}\lambda_2 \mu^p$ is nontrivial in $\pi_1(X \cup_r (S^1 \times D^2))$ for all $r \in \mathbb{Q}$. 
Since $\pi_1(X \cup_r (S^1 \times D^2))$  injects into $\pi_1(K(r))$, 
we have $\mathcal{S}(\lambda_1^{-1}\lambda_2\mu^p) = \emptyset$. 
\end{proof}

\medskip

In the above discussion, 
choosing $p = \pm 1$, the persistent element $\lambda_1^{-1}\lambda_2\mu^p$ is a homological generator. 
Hence, even a persistent element representing a homological generator need not be a pseudo-meridian.

\medskip

For some hyperbolic knots, 
a persistent element that is a homological generator may not be a pseudo-meridian. 
Here we give such examples. 

\begin{example}
\label{persistent_not_pseudo-meridian}
Let $K$ be the figure-eight knot with
\[
G(K)=\langle t,a \mid ta^2t=ata^{-1}ta\rangle,
\]
where $t$ is a meridian and $a$ lies in the commutator subgroup.

Let $g=t^{-1}(a^2ta)t^2(a^2ta)^{-1}$.
Then $g$ represents a homological generator, and $a^2ta\ne 1$ in $G(K)$.
By Theorem \ref{persistent_BS}, $g$ is a persistent element.

However, $g$ is not a pseudo-meridian, since the quotient group $G(K)/ \langle \! \langle g \rangle \!\rangle$ has order $336$, which is confirmed by GAP \cite{GAP}. 
\end{example}

The knot in Example~\ref{persistent_not_pseudo-meridian} has the knot group generated by two elements in which 
one is a meridian. 
Such knots admit a pseudo-meridian. 

\medskip

\begin{example}
Let $K$ be a knot whose knot group is generated by a meridian $\mu$ and $y$. 
Then $g_m = \mu^{1-m} y \mu^m y^{-1}$ is a pseudo-meridian. 
\end{example}

\begin{proof}
In Case 3 of the  proof of Theorem~\ref{existence_persistent_elements}, 
we found that $g_m=\mu^{1-m}y\mu^m y^{-1}$ is persistent for a hyperbolic knot,
if $m>1$ and $[\mu, y]\ne 1$.
We remark that
this element $g_m$ can be a pseudo-meridian.
Suppose that $G(K)$ is generated by two meridians $\mu$ and $y$.
Since $G(K)$ is not abelian, $[\mu, y]\ne 1$.
Then the quotient group $G/\langle\!\langle g_m \rangle\!\rangle$
is perfect, so is trivial by Miller-Schupp \cite{MSc}.
\end{proof}

\section{Further questions}

This paper's main focus is to clarify the extent to which persistent elements are abundant.  
In this final section, we propose some questions.  

\subsection{Persistent subgroup}
Theorem~\ref{characterization} characterizes subgroups which contains a persistent element, 
and proves that expected subgroups always contain persistent elements. 
We say that a subgroup $H \subset G(K)$ is \textit{persistent} if its nontrivial elements are persistent. 
Building upon Theorem~\ref{existence_persistent_elements},  
we prove that the knot group $G(K)$ admits a cyclic persistent subgroup if and only if $K$ has no finite surgery; see \cite{IMT_cyclic}.  
So, for instance, any torus knot group do not admit persistent subgroups.  
Concerning non-cyclic persistent subgroups, 
we may prove that most satellite knots admits a persistent free group of rank two, and hence infinite rank. 
On the other hand, to the best of our knowledge, no such examples are known for hyperbolic knots. 

So we would like to ask: 

\begin{question}\footnote{We would like to thank Michel Boileau for directing our attention to persistent free group of rank two.}
\label{non-cyclic_persistent_subgroup}
Let $K$ be a hyperbolic knot. 
Does $G(K)$ admit a persistent free group of rank two? 
\end{question}

\bigskip

\subsection{Density of persistent elements}

To formulate that 
whether persistent element is common or not 
we introduce the following
\footnote{We would like to thank David Futer for suggesting this perspective.}.

Fix a finite generating set $X$ of a knot group $G(K)$ and 
we denote the word length of a nontrivial element $g \in G(K)$ by $\ell_X(g)$.
For $n \in \mathbb{Z}_{\geq 0}$, we denote by $B_X(n)$ the ball of radius $n$ centered at the identity with respect to the word length $\ell_X$.

For a subset $S$ of $G(K)$, the \emph{natural} (or  \emph{asymptotic}) \emph{density}  with respect to $X$ is defined as 

\[
\delta_{X}(S) = \limsup_{n \to \infty}\frac{ \#(S \cap B_X(n))}{ \# B_X(n)} .
\]

A subset $S$ is called \emph{generic} (resp. \emph{negligible}) with respect $X$ if $\delta_X(S)=1$ (resp. $\delta_X(S)=0$).

The value of natural density may depend on a choice of generating set $X$, and the behavior of the natural density is not `natural' as its name suggests. We refer to \cite{BurilloVentura} for several strange properties of natural density.

It is interesting problem to explore the natural density of the set of persistent elements.

\begin{question}
\label{prob}
Let $\mathcal{P}_K$ be the set of persistent elements in $G(K)$.
\begin{enumerate}
\renewcommand{\labelenumi}{(\roman{enumi})}
\item Is $\mathcal{P}_K$ a net ? i.e., there exists a constant $C$ and a finite generating set $X$ of $G(K)$ such that for every $g \in G(K)$, there exists $p \in \mathcal{P}_K$ with $\ell_X(g^{-1}p) \leq C$.
\item Is $\delta_X(\mathcal{P}_K)>0$ for a suitable finite generating set $X$?
\item Is $\delta_X(\mathcal{P}_K)=1$ for a suitable finite generating set $X$?
\end{enumerate}
\end{question}

We remark that if $\mathcal{P}_K$ is a net, then the property that $\mathcal{P}_K$ is non-negligible (i.e. $\delta_X(\mathcal{P}_K)>0$)  is independent of a choice of generating set $X$ by \cite[Proposition 2.1]{BurilloVentura}, so the answer of (ii) is independent of a choice of finite generating set $X$.

Note that $\mathcal{P}_K$ is written as the set 
$\{ g \in G(K) \mid |\mathcal{S}(g)| = 0 \}$. 
It follows from the Realization property mentioned in Subsection~\ref{subsection:background} that for any finite subset $\mathcal{R} \subset \mathbb{Q}$, we have an element $g \in G(K)$ such that $\mathcal{S}(g) = \mathcal{R}$ for a large class of hyperbolic knots $K$. 
Furthermore, there are infinitely many, mutually non-conjugate such elements; see \cite[Theorem~1.10]{IMT_realization}.  
Related to Question~\ref{prob}, it may be plausible to ask: 

\begin{question}
\label{prob_m}
\[
\lim_{m \to \infty}\delta_X(\{ g \in G(K) \mid |\mathcal{S}(g)| \ge m\}) = 0
\]
for any \(a suitable\) generating set $X$?
\end{question}

\bigskip

\section*{Acknowledgements}
TI has been partially supported by JSPS KAKENHI Grant Number 21H04428 and 23K03110.

KM has been partially supported by JSPS KAKENHI Grant Number 25K07018, 21H04428, 23K03110, 23K20791 and Joint Research Grant of Institute of Natural Sciences at Nihon University for 2025. 

MT has been partially supported by JSPS KAKENHI Grant Number JP25K07004.

\end{document}